\documentclass[10.5pt]{amsart}

\usepackage[papersize={21cm,29.7cm},top=4.5cm,bottom=4.5cm,left=3.7cm,right=3.7cm,marginparsep=0.3cm,marginparwidth=0cm,includefoot]{geometry}
\usepackage{amssymb}
\usepackage[backend = biber]{biblatex}
\usepackage{bbm}
\usepackage[english]{babel}
\usepackage[latin1]{inputenc}
\usepackage[T1]{fontenc}
\usepackage[dvipsnames]{xcolor}
\usepackage{hyperref}
\hypersetup{
    colorlinks = true,
    linkcolor = Blue,
    citecolor = Blue,
    urlcolor = Blue
    }
\usepackage{ifpdf}
\ifpdf
  \usepackage[pdftex]{graphicx}
  \usepackage{epstopdf}
\else
  \usepackage[dvips]{graphicx}
\fi
\usepackage{graphicx}
\usepackage{cleveref}
\usepackage[all]{xy}
\usepackage{csquotes}
\usepackage{dynkin-diagrams}
\usepackage{etoolbox}

\makeatother

\theoremstyle{plain}
\newtheorem{theorem}{Theorem}[section]

\newtheorem{lemma}[theorem]{Lemma}
\newtheorem{proposition}[theorem]{Proposition}
\newtheorem{corollary}[theorem]{Corollary}
\newtheorem{question}[theorem]{Question}

\theoremstyle{definition}
\newtheorem{definition}[theorem]{Definition}

\newtheorem{convention}[theorem]{Convention}

\newtheorem{remark}[theorem]{Remark}

\theoremstyle{remark}

\newcommand{\Z}{\mathbb Z}
\newcommand{\Q}{\mathbb Q}
\newcommand{\C}{\mathbb C}
\newcommand{\R}{\mathbb R}

\newcommand{\V}{\mathbb V}

\newcommand{\p}{\mathbb P}
\newcommand{\s}{\mathbb S}

\newcommand{\mult}{\mathbb G}
\newcommand{\G}{\mathbf G}
\newcommand{\mon}{\mathbf H}
\newcommand{\M}{\mathbf M}

\newcommand{\fix}{\mathbf{Fix}}
\newcommand{\stab}{\mathbf{Stab}}
\newcommand{\hl}{\mathrm{HL}}

\newcommand{\HLM}{\mathrm{HL}(S, \V^\otimes, \M)}

\newcommand{\e}{\ast}

\newcommand{\im}{\mathrm{im} \hspace{0.05cm}}

\newcommand{\an}{\mathrm{an}}

\newcommand{\Hom}{\mathrm{Hom}}

\newcommand{\quo}{\backslash}

\newcommand{\gl}{\mathbf{GL}}

\newcommand{\res}{\mathrm{Res}}

\newcommand{\g}{\mathfrak{g}}

\newcommand{\ph}{\varphi}

\newcommand{\mt}{\mathfrak{m}}

\newcommand{\abel}{\mathfrak{a}}

\newcommand{\ad}{\mathrm{ad}}

\newcommand{\typ}{\mathrm{typ}}

\newcommand{\der}{\mathrm{der}}
\newcommand{\trans}{\mathrm{trans}}

\newcommand{\ct}{\mathrm{CT}}
\newcommand{\prim}{\mathrm{prim}}
\newcommand{\End}{\mathrm{End}}

\title{Global variations of Hodge structures of maximal dimension}
\author{Nazim Khelifa}
\date{\today}

\begin{document}
\maketitle
\begin{abstract}
We derive a new bound on the dimension of images of period maps of global pure polarized integral variations of Hodge structures with generic Hodge datum of level at least $3$. When the generic Mumford-Tate domain of the variation is a period domain parametrizing Hodge structures with given Hodge numbers, we prove that the new bound is at worst linear in the Hodge numbers, while previous known bounds were quadratic. We also give an example where our bound is significantly better than previous ones and sharp in the sense that there is a variation of geometric origin whose period image has maximal dimension (i.e. equal to the new bound).
\end{abstract}
\begin{center}
\tableofcontents
\end{center}
\section{Introduction}
\subsection{Context}
Let $\V$ be a polarized integral variation of pure Hodge structures on a smooth, irreducible and quasi-projective complex algebraic variety $S$. Denote by $(\G,D)$ its generic Hodge datum, by $\Gamma \subset \G(\Q)$ an arithmetic lattice containing the image of the monodromy representation associated to $\V$, and by $\ph : S^\an \rightarrow \Gamma \quo D$ the resulting period map. Griffiths transversality imposes differential-geometric constraints on the image of $\ph$, namely it must be tangent to a non-trivial distribution of $\Gamma \quo D$ which is non-integrable in general. As discussed in \cite{carltolab2}, these constraints can be translated into bounds on the dimension of the image of $\ph$. Recall that any point $x \in D$ seen as a morphism of real algebraic groups $x : \s := \res_{\C/\R}(\mult_{m,\C}) \rightarrow \G_\R$ induces, by post-composition with the adjoint representation, a $\Q$-Hodge structure of weight zero on the Lie algebra $\g$ of $\G$ whose Hodge decomposition will be denoted by
\[
\g_\C := \g \otimes_\Q \C = \bigoplus_{k \in \Z} \g_x^{-k,k}.
\] 
The corresponding real Hodge structure on $\g_\R$ doesn't depend on the choice of $x\in D$ up to isomorphism. The result of Carlson and Toledo \cite[Prop 5.3]{carltolab2} is now that
\[
\dim \ph(S^\an) \leqslant m_\ct(\G,D)
\]
where $m_{\ct}(\G,D) := \max\{\dim_\C \abel \hspace{0.1cm} : \hspace{0.1cm} \abel \subset \g^{-1,1}, [\abel, \abel] = 0\}$. We will refer to this bound as the \textit{Carlson-Toledo bound}.

When $D$ is a period domain, the Carlson-Toledo bound can be computed explicitly and is quadratic in the Hodge numbers \cite[(1.5)]{ckt}. The question of the sharpness of this bound attracted much interest in the past (see. \cite{carlson}, \cite{ckt}, \cite{carlsimp}, \cite{carltolab3}, \cite{mayer}). These works have two essential limitations from an algebro-geometric perspective:
\begin{enumerate}
\item Pure polarized integral variations of Hodge structures that arise as local systems of primitive cohomology groups of fibers of families of projective complex varieties usually have a base which is a (quasi-)projective algebraic variety, and their global nature is widely known to constrain their behaviour. On the other hand, all the aforementionned works are infinitesimal in nature, and therefore bound the dimension of germs of variations of Hodge structures forgetting their global nature. Therefore one could hope that better dimension bounds exist when one takes into account global properties of the variation.
\item Almost all known examples of global variations which are maximal i.e. have period image of dimension equal to the Carlson-Toledo bound are horizontal totally geodesic subvarieties of $\Gamma \quo D$ arising as arithmetic quotients of hermitian symmetric subdomains of $D$, which is not fully satisfying as one would like to call a variation "maximal" when its period image has dimension equal to the Carlson-Toledo bound of its generic Hodge datum. To paraphrase Carlson and Toledo \cite{nongeodesic}, such examples "are defined by Lie theory rather than by algebraic geometry". 
\end{enumerate}
The purpose of this work is to give a general method to improve existing bounds by taking into account the above two remarks. Namely we adress the following:
\begin{question}\label{mainq}
Let $(\G,D)$ be a Hodge datum.
\begin{itemize}
\item Can one improve the bound $m_{\ct}(\G,D)$ by taking into account the assumptions of quasi-projectivity of the base and of Hodge genericity?
\item Can one exhibit a pure polarized integral variation of Hodge structures on a quasi-projective base (and not only a germ of variation) with generic Hodge datum $(\G,D)$, whose period image dimension attains this bound?
\end{itemize}
\end{question}

The answer naturally depends on the level of the Hodge datum, in the following sense introduced in \cite{bku}:
\begin{definition}
Let $(\G,D)$ be a Hodge datum such that the derived subgroup $\G^\der$ of $\G$ is $\Q$-simple. Denote by $\g$ its Lie algebra, fix $x \in D$ and denote by 
\[
\g\otimes_\Q \C = \bigoplus_{k \in \Z} \g_x^{-k,k}
\] 
the induced Hodge decomposition. The \textit{level} of $(\G,D)$ is the largest integer $k$ such that $\g_x^{-k,k} \neq \{0\}$. It doesn't depend on the choice of $x$.
\end{definition}

The answer is obviously positive when $(\G,D)$ is a Shimura datum i.e. has level $1$, as Griffiths transversality imposes no condition there, and connected Shimura varieties are quasi-projective algebraic varieties. In \cite{nongeodesic}, Carlson and Toledo exhibit a variation of hypersurfaces in a weighted projective space (hence of geometric origin) whose generic Hodge datum has level $2$ and whose period image has dimension equal to the Carlson-Toledo bound of its generic Hodge datum. 

\subsection{Bounds in higher level} The purpose of this note is to explain how the Zilber-Pink paradigm developped in \cite{bku} can be used to answer by the positive to the first part of Question \ref{mainq} when the level of the generic Hodge datum is at least $3$. Let us start with the:
\begin{definition}
Let $(\G,D)$ be a Hodge datum. The \textit{Hodge locus bound} associated to $(\G,D)$ is defined as
\[
m_\hl(\G,D) = \inf_{(\M,D_M) \subsetneq (\G,D)} \Big[ \dim D - \dim D_M \Big] - 1
\]
where the infimum is taken over all strict Hodge sub-data of $(\G,D)$. In particular $m_\hl(\G,D) = + \infty$ if $(\G,D)$ doesn't contain any strict Hodge sub-datum.
\end{definition}
Our general result reads as follows:
\begin{theorem}\label{main}
Let $(\G,D)$ be a Hodge datum of level at least $3$ such that $\G^\der$ is $\Q$-simple. Let $\V$ be a pure polarized integral variation of Hodge structures, with generic Hodge datum $(\G,D)$, on a smooth, irreducible and quasi-projective complex algebraic variety $S$. Let $\Gamma \subset \G(\Q)$ be an arithmetic lattice containing the image of the monodromy representation associated to $\V$ and $\ph : S^\an \rightarrow \Gamma \quo D$ be the corresponding period map. Then,
\[
\dim \ph(S^\an) \leqslant m_\hl(\G,D).
\]
\end{theorem}
For this result to be useful in studying the first part of Question \ref{mainq} for some Hodge datum $(\G,D)$ one has to compare $m_\ct(\G,D)$ to $m_\hl(\G,D)$, which can be done by estimating explicitly $m_\hl(\G,D)$ in some cases. We do it here for Hodge data such that $D$ is a period domain, to which we refer as \textit{period data} (see Section \ref{perdat} for details, especially on how to see the level of a period datum on its associated Hodge numbers).

In \cite[(1.5)]{ckt}, the Carlson-Toledo bound is computed explicitly for period data in terms of the Hodge numbers of the structures that they parametrize. We refer to \textit{op. cit.} for the precise formula giving $m_\ct(\G,D)$, but we emphasize at this point on the fact that it is given by a quadratic function of the Hodge numbers. We prove that for period data the "global" Hodge locus bound is at worst linear in the Hodge numbers, hence quickly better than the "local" Carlson-Toledo bound as the Hodge numbers grow:
\begin{theorem}\label{computation}
Let $(\G,D)$ be a period datum with associated Hodge numbers $(h^{p,q})_{p+q = w}$ such that $h^{w,0} \neq 0$ and $h^{p,q} = 0$ for $p<0$ or $q<0$. Let $\V$ be a pure polarized integral variation of Hodge structures, with generic Hodge datum $(\G,D)$, on a smooth, irreducible and quasi-projective complex algebraic variety $S$. Let $\Gamma \subset \G(\Q)$ be an arithmetic lattice containing the image of the monodromy representation associated to $\V$ and $\ph : S^\an \rightarrow \Gamma \quo D$ be the corresponding period map.
\begin{itemize}
\item[$(a)$] If $(\G,D)$ parametrizes Hodge structures of even weight $w = 2n$ with $n \geqslant 2$ and $h^{n,n}\neq 0$, one has
\[
\dim \ph(S^\an) \leqslant \Big(\sum_{i = 1}^n h^{n-i,n+i}\Big) - 1.
\]
\item[$(b)$] If $(\G,D)$ parametrizes Hodge structures of odd weight $w = 2n+1$ with $n \geqslant 1$. For any non-zero Hodge number $h^{r,s}$ one has
\[
\dim \ph(S^\an) \leqslant \Big(2 \sum_{i = 0}^n h^{n-i, n+i+1}\Big) - 2 - h^{r,s}.
\]
\end{itemize}
\end{theorem}
\subsection{Sharpness} The new bounds exhibited above give an efficient way of answering positively to the first part of Question \ref{mainq} for a given period datum of level at least $3$. It suffices to compute the Hodge locus bound using Theorem \ref{computation} and to compare it to the Carlson-Toledo bound using \cite[(1.5)]{ckt}. We exhibit here a period datum for which the Hodge locus bound is sharp and better than the Carlson-Toledo bound.

Let $f_4 : \mathcal{X}_4 \rightarrow U_4$ be the universal family of non-singular sextic fourfolds in $\p^5(\C)$ and $\V_4$ be the associated variation on primitive middle cohomology (see Section \ref{sharp} for details). Let $(\G_{4},D_{4})$ be its generic Hodge datum. By \cite[Cor. 5.5]{weil} it is a period datum, and it parametrizes weight $4$ polarized $\Q$-Hodge structures with Hodge numbers $h^{2,2} = 1755$, $h^{3,1} = 426$ and $h^{4,0} = 1$. One has
\[
m_\hl(\G_{4}, D_{4}) \leqslant h^{3,1} + h^{4,0} - 1 = 426,
\]
and
\[
m_\ct(\G_{4}, D_{4}) = \max \{h^{3,1}h^{4,0}, \frac{1}{2} h^{3,1}h^{2,2} \} = \max \{426, 373815\} = 373815.
\]
In particular, $m_\hl(\G_{4}, D_{4}) < m_\ct(\G_{4}, D_{4})$. Furthermore, one has the following which shows that the answer to the full Question \ref{mainq} is positive for $(\G_4,D_4)$.
\begin{proposition}\label{sharpfourfolds}
The period map $\ph_4$ associated to $\V_4$ has image of dimension
\[
\dim \ph_4(U_4^\an) = 426 = m_\hl(\G_{4},D_{4}).
\]
In particular, it is maximal for the dimension among pure polarized integral variations of Hodge structures over a smooth, irreducible and quasi-projective base which have generic Hodge datum $(\G_4,D_4)$.
\end{proposition}
Let us compare this result with the general result of Carlson-Donagi in \cite{carldon} which applies to the above variation of hypersurfaces. It states that for a non-singular sextic fourfold $X$, the germ of the hypersurface variation at $X$ in $D_4$ is \textit{maximal for the inclusion} among germs of polarized variations of Hodge structures in $D_4$. However, it doesn't ensure that it is \textit{maximal for the dimension} among germs, and in fact it is not, as using the methods of \cite{ckt} one can construct a germ of dimension $373815$ as a sum of appropriate root subspaces of the complexified Lie algebra of $\G_4$. Our result is in some sense complementary to theirs, as it states that as a pure polarized integral variation of Hodge structures with generic Hodge datum $(\G_4,D_4)$ on a quasi-projective base, it is maximal for the dimension. In particular, it implies that the aforementionned germs of variations of bigger dimension don't integrate to Hodge generic variations on a quasi-projective base.
\subsection{Notations} Throughout the text, we will use the following notations:
\begin{itemize}
\item \textit{Algebraic groups :} We will denote $\Q$-algebraic groups in bold letters (e.g. $\G$,$\M$,$\mon$) and for a reductive $\Q$-group $\G$, we will denote by $\G^\der$ its derived subgroup and $\G^\ad$ the adjoint group of $\G$. We will use roman letters for real Lie groups. The multiplicative group over a field $k$ will be denoted $\mult_{m,k}$. For a rational representation of a $\Q$-algebraic group $\G \rightarrow \gl(V)$, and for a vector $v \in V$, we will denote by $\fix_\G(v)$ the fixator in $\G$ of $v$ and $\stab_\G(v)$ the stabilizer in $\G$ of the line $\Q v$. 
\item \textit{Scalar extensions :} If $V$ is a $\Q$-vector space (resp. $\G$ is a $\Q$-algebraic group) and $\Q \subset k \subset \C$ is a field, we will denote by $V_k$ (resp. $\G_k$) the $k$-vector space (resp. $k$-algebraic group) induced by scalar extension.
\item \textit{Dimensions :} If $V$ is a $k$-vector space over a field $k$, we will denote its $k$-dimension by $\dim_k V$. If $X$ is an equidimensional complex manifold (this includes in particular Mumford-Tate domains), we will denote by $\dim X$ its dimension at any smooth point. If $G$ is a real Lie group, we will denote by $\dim_\R G$ its dimension as a real-analytic manifold.
\end{itemize}
\subsection{Aknowledgements} I thank my advisor Emmanuel Ullmo for several discussions, suggestions and comments. I also thank Paul Brommer-Wierig for his careful reading and comments on an earlier version, as well as Gregorio Baldi, Joshua Lam and David Urbanik for useful related discussions.
\section{(Non)-emptiness of transverse Hodge loci and dimension of period images}
This section is devoted to proving Theorem \ref{main} using recent advances on understanding the distribution of Hodge loci. Throughout the section we fix a pure polarized integral variation of Hodge structures $\V$ on a smooth, irreducible and quasi-projective complex algebraic variety $S$. Possibly replacing $S$ by a finite étale cover, which will be harmless for our considerations, we can assume that the image of the monodromy representation associated to $\V$ is torsion-free. We denote by $(\G,D)$ its generic Hodge datum, by $\Gamma \subset \G(\Q)$ an arithmetic torsion-free lattice containing the image of the monodromy representation associated to $\V$, and by $\ph : S^\an \rightarrow \Gamma \quo D$ the resulting period map. We assume that the derived subgroup $\G^\der$ of $\G$ is $\Q$-simple.
\subsection{Transverse Hodge loci} We start by defining a notion of so-called transverse special subvarieties of $S$ for $\V$, which is a weakened, purely geometric version of typicality.
\begin{definition}
Let $Z$ be an irreducible algebraic subvariety of $S$ and $(\M,D_M)$ a Hodge sub-datum of $(\G,D)$. We say that $Z$ is \textit{defined by $(\M,D_M)$} if it is maximal for the inclusion among irreducible subvarieties of $S$ whose generic Hodge datum is a sub-datum of $(\M,D_M)$.
\end{definition}
A special subvariety of $S$ for $\V$ is always defined by its generic Hodge datum, but the usefulness of this notion is that it can also be defined by a Hodge sub-datum of $(\G,D)$ which contains strictly its generic Hodge datum. In this spirit it is natural to define the following weakening of the notion of typicality:
\begin{definition}
Let $Z$ be a special subvariety of $S$ for $\V$. We say that $Z$ is \textit{transverse} if it can be defined by a (strict) Hodge sub-datum $(\M,D_M)$ of $(\G,D)$ such that:
\[
\dim \ph(S^\an) + \dim D_M - \dim D = \dim \ph(Z^\an).
\]
We then say that $(\M,D_M)$ \textit{defines $Z$ transversally}. The \textit{transverse Hodge locus of type $\M$} is the union $\HLM_\trans$ of special subvarieties of $S$ for $\V$ which are transversally defined by some $\G(\Q)$-conjugate of $(\M,D_M)$.
\end{definition}
Obviously, a typical special subvariety of $S$ for $\V$ is automatically transverse but the resulting containment $\HLM_\typ \subset \HLM_\trans$ is strict in general as a special subvariety defined transversally by some Hodge datum $(\M,D_M)$ may in principle have a generic Hodge datum strictly smaller than $(\M,D_M)$ (in which case it is atypical).
\subsection{Criteria for (non)-existence of transverse Hodge loci} The purpose of this section is to explain how the results \cite[Thm. 3.4]{es}, \cite[Thm. 1.9 (i)]{ku} and \cite[Thm. 3.3]{bku} read in terms of transverse Hodge loci. The first two give a criterion for the existence of transverse special subvarieties:
\begin{proposition}[{\cite[Thm. 3.4]{es}, \cite[Thm. 1.9 (i)]{ku}}]\label{existence}
Let $(\M,D_M)$ be a strict Hodge sub-datum of $(\G,D)$. The following are equivalent:
\begin{itemize}
\item[$(i)$]$(\M,D_M)$ is $\V$-admissible:
\[\dim D_M + \dim \ph(S^\an) - \dim D \geqslant 0;\]
\item[$(ii)$] $\HLM_\trans\neq \emptyset$;
\item[$(iii)$] $\HLM_\trans$ is analytically dense in $S^\an$.
\end{itemize}
\end{proposition}
\begin{proof}
Clearly one has the chain of implications $(iii)\implies(ii) \implies (i)$ so it suffices to prove $(i) \implies (iii)$. Assume that $(\M,D_M)$ is $\V$-admissible. The proof of \cite[Thm. 1.9 (i)]{ku} produces a set of special subvarieties of $S$ for $\V$ whose union is analytically dense in $S^\an$ and which correspond to intersections of expected dimension in $\Gamma \quo D$ of the projection to $\Gamma \quo D$ of some $\G(\Q)$-translate of $D_M$ with $\ph(S^\an)$. This precisely means that they are transversally defined by some $\G(\Q)$-conjugate of $(\M,D_M)$. Therefore $\HLM_\trans$ is analytically dense in $S^\an$.
\end{proof}
Similarly, the proof of \cite[Thm. 3.3]{bku} gives the following slightly stronger result, which is a criterion for non-existence of transverse special subvarieties:
\begin{proposition}[{\cite[Thm. 3.3]{bku}$+\varepsilon$}] \label{nonexistence}
Assume that $(\G,D)$ has level at least $3$. Then for any strict Hodge sub-datum $(\M,D_M)$ of $(\G,D)$, the transverse Hodge locus $\HLM_\trans$ of type $\M$ is empty.
\end{proposition}
\begin{proof}
We simply adapt the proof in \cite[Sect. 7.1]{bku} to our context. Let $(\M,D_M)$ be a strict Hodge sub-datum of $(\G,D)$ such that some special subvariety $Z$ of $S$ is transversally defined by $(\M,D_M)$. Seeking for a contradiction, we want to prove that $(\M,D_M) = (\G,D)$. Let $\mt$ and $\g$ be the Lie algebras of $\M$ and $\G$. Let $x \in \ph(Z^\an)$ and for $i \in \Z$ denote by $\mt_x^{-i,i}$ and $\g_x^{-i,i}$ the pieces of the Hodge decomposition induced by $x$ on $\mt_\C$ and $\g_\C$. One has:
\[ \dim \bigoplus_{i > 0} \mt_x^{-i,i} - \dim \ph(Z^\an) = \dim \bigoplus_{i>0} \g_x^{-i,i} - \dim \ph(S^\an). \]
Because the intersection takes place inside of the horizontal distribution, this equality of dimension splits into
\[ \dim \mt_x^{-1,1} - \dim \ph(Z^\an) = \dim \g_x^{-1,1} - \dim \ph(S^\an)\]
and
\[ \dim \bigoplus_{i>1} \mt_x^{-i,i} = \dim \bigoplus_{i>1} \g_x^{-i,i}.\]
As $\mt_x^{-i,i} \subset \g_x^{-i,i}$ this last equality of dimensions implies that $\mt_x^{-i,i} = \g_x^{-i,i}$ for each $i >1$ and therefore, by Hodge symmetry,
\begin{equation}\label{eqmg}
\forall |i| \geqslant 2, \hspace{0.2cm} \mt_x^{-i,i} = \g_x^{-i,i}.
\end{equation}
The existence of $Z$ implies in particular that the period map has positive dimensional image, hence that the algebraic monodromy group $\mon$ is non-trivial. Indeed, if it was trivial, the monodromy representation would be trivial (recall that we assumed that it has torsion-free image) so that the local system underlying $\V$ would be trivial. It would then follow from \cite[Cor. 7.23]{schmid} that the variation of Hodge structures on $\V$ is trivial contradicting the positive dimensionality of the image of the period map. As we assumed that $\G^\der$ is $\Q$-simple, André's theorem \cite[Thm. 1]{andre} implies that $\mon = \G^\der$. Then \cite[Prop. 7.4]{bku} ensures that the $\Q$-Hodge-Lie algebra $\g$ is generated in level $1$ (in the sense of \cite[Def. 7.3]{bku}). As it is of level greater or equal to $3$ by assumption, \cite[Prop. 7.5]{bku} combined to (\ref{eqmg}) gives that $\mt = \g$. This finishes the proof by the connectedness of $\M$ and $\G$ and the inclusion $D_M \subset D$.
\end{proof}
\subsection{Period dimensions in higher level} We can now, combining Propositions \ref{existence} and \ref{nonexistence}, prove our first main result Theorem \ref{main}.
\begin{proof}[Proof of Theorem \ref{main}] Assume that $(\G,D)$ is of level at least $3$. If there is no strict Hodge sub-datum of $(\G,D)$, one has $m_\hl(\G,D) = + \infty$ and there is nothing to prove. Otherwise, pick one and denote it by $(\M,D_M)$. Proposition \ref{nonexistence} ensures that $\HLM_\trans$ is empty. The implication $(i) \implies (iii)$ in Proposition \ref{existence} then says that $(\M,D_M)$ cannot be $\V$-admissible, i.e.
\[
\dim \ph(S^\an) + \dim D_M - \dim D < 0
\]
which rewrites
\[
\dim \ph(S^\an) \leqslant \dim D - \dim D_M - 1.
\]
As this is true for every strict sub-datum $(\M,D_M) \subsetneq (\G,D)$, one gets the desired inequality:
\[
\dim \ph(S^\an) \leqslant \min_{(\M,D_M) \subsetneq (\G,D)} \Big[ \dim D - \dim D_M \Big] - 1 = m_\hl(\G,D).
\]
\end{proof}
\begin{remark}
The reason for working with the notion of transverse special subvarieties (instead of the usual notion of typicality) lies in the fact that the analog of Proposition \ref{existence} for the typical Hodge locus is not known when the corresponding loci are points, because of the inability to exclude the eventuality of them being atypical on account on their Mumford-Tate group being properly contained in the appropriate translate of $\M$. The point here is that one can work with the weaker notion of transverse Hodge loci as the generalization Proposition \ref{nonexistence} of \cite[Thm. 3.3]{bku} easily holds.
\end{remark}
\section{The Hodge locus bound for period data}
To be able to use Theorem \ref{main} to adress Question \ref{mainq}, we are left with proving the estimates in Theorem \ref{computation}.
\subsection{Period data}\label{perdat}
We assume given a tuple $\Sigma = (V, \psi, w, (h^{p,q}))$ which consists of
\begin{itemize}
\item a finite dimensional non-trivial $\Q$-vector space $V$;
\item a positive integer $w$;
\item a set of integers $h^{p,q}$ indexed by $\{(p,q) \in \Z^2 \hspace{0.1pt} : \hspace{0.1pt} p+q = w\}$, such that $h^{p,q} = h^{q,p}$ for every $(p,q)$, $h^{p,q} = 0$ for all but finitely many $(p,q)$ and $\dim_\Q V = \sum_{p+q = w} h^{p,q}$;
\item a non-degenerate $(-1)^w$-symmetric bilinear form $\psi : V \times V \rightarrow \Q$.
\end{itemize}
Let $D_\Sigma$ be the set of Hodge structures of weight $w$ and Hodge numbers $(h^{p,q})$ on $V$ which are polarized by $\psi$. Assume that $D_\Sigma$ is non-empty, which amounts to a sign condition on the hermitian form induced by $\psi$ on $V_\C$. Let $\G_\Sigma$ be the reductive $\Q$-group $\mathbf{GAut}(V,\psi)$ of similitudes of $(V,\psi)$. A straightforward verification of the axioms in \cite[Def. 3.1]{geoao} shows that the pair $(\G_\Sigma, D_\Sigma)$ is a Hodge datum. We will refer to $(\G_\Sigma, D_\Sigma)$ as the \textit{period datum associated to $(V, \psi, w, (h^{p,q}))$}. We will say that it is \textit{orthogonal} if $w$ is even and \textit{symplectic} if $w$ is odd. Remark that by possibly changing the weight $w$ and Tate-twisting, any period datum is isomorphic to one associated to a tuple satisfying the following assumption which we take as a convention in the sequel:
\begin{convention}\label{conv}
We will assume that tuples $(V, \psi, w, (h^{p,q}))$ defining period data are such that $h^{w,0} > 0$ and $h^{p,q} = 0$ for $p<0$ or $q<0$.
\end{convention}
For the convenience of the reader and to ease the application of our bounds, we now give explicit formulae to compute the level of a period datum $(\G,D)$ associated to a tuple $(V, \psi, w, (h^{p,q}))$. Let $x \in D$ be any point, and recall that, as explained in the introduction, one gets a $\Q$-Hodge structure of weight $0$ on the Lie algebra $\g$ of $\G$, whose associated Hodge decomposition we denote by
\[
\g_{\C} = \bigoplus_{k \in \Z} \g_{x}^{-k,k}.
\]
Different choices of $x \in D$ give isomorphic real Hodge structures on $\g_{\R}$, and in particular, the Hodge numbers $h_{\inf}^{k} := \dim_\C \g_{x}^{-k,k}$ don't depend on the choice of $x$ and satisfy Hodge symmetry $h_{\inf}^{k} = h_{\inf}^{-k}$. We will call them the \textit{infinitesimal Hodge numbers} of $(\G,D)$. The following is a direct consequence of the block decompositions presented in \cite[Sect. 3]{ckt}:
\begin{proposition}\label{infhodgenum}
Let $(\G,D)$ be the period datum associated to a tuple $(V, \psi, w, (h^{p,q}))$. 
\begin{itemize}
\item Assume $(\G,D)$ is orthogonal, i.e. $w = 2n$ is even. The infinitesimal Hodge numbers are
\[
h_{\inf}^{2k} = \sum_{j = 0}^{n-k-1} h^{2n-j,j}h^{2n-j-2k, j+2k} + \frac{1}{2}h^{n+k,n-k}(h^{n+k,n-k}-1)
\]
and
\[
h_{\inf}^{2k+1} = \sum_{j = 0}^{n-k-1} h^{2n-j,j}h^{2n-j-2k-1, j+2k+1}
\]
for $k \geqslant 0$.
\item Assume $(\G,D)$ is symplectic, i.e. $w = 2n+1$ is odd. The infinitesimal Hodge numbers are
\[
h_{\inf}^{2k} = \sum_{j = 0}^{n-k} h^{2n+1-j,j}h^{2n+1-j-2k, j+2k}
\]
and
\[
h_{\inf}^{2k+1} = \sum_{j = 0}^{n-k-1} h^{2n+1-j,j}h^{2n-j-2k, j+2k+1} + \frac{1}{2}h^{n+k+1,n-k}(h^{n+k+1,n-k}+1)
\]
for $k \geqslant 0$.
\end{itemize}
\end{proposition}
An immediate consequence is the following characterization of period data of level at least $3$. We emphasize once again that we are here using Convention \ref{conv}.
\begin{corollary}\label{levelcrit}
Let $(\G, D)$ be a period datum associated to the tuple $(V, \psi, w, (h^{p,q}))$ and assume that $\G^\der$ is $\Q$-simple. Then $(\G,D)$ is of level at least $3$ if, and only if one of the following cases occur:
\begin{enumerate}
\item $w$ is odd and $w \geqslant 3$;
\item $w = 2n$ is even with $n \geqslant 2$ and at least one of the following holds:
\begin{enumerate}
\item[(2.a)] for some $2 \leqslant k \leqslant n$ one has $h^{n+k,n-k} > 1$;
\item[(2.b)] for some $2 \leqslant k \leqslant n$ and some $0 \leqslant j \leqslant n-k-1$ one has \[ h^{2n-j,j}h^{2n-j-2k, j+2k} > 0;\]
\item[(2.c)] for some $1 \leqslant k \leqslant n-1$ and some $0 \leqslant j \leqslant n-k-1$ one has \[h^{2n-j,j}h^{2n-j-2k-1, j+2k+1}>0.\]
\end{enumerate}
 \end{enumerate}
\end{corollary}
\begin{proof}
It is clear from the formulae in Proposition \ref{infhodgenum} that under Convention \ref{conv} the level of a period datum with $\Q$-simple derived Mumford-Tate group is lower or equal to its weight $w$. In particular, if $w$ is odd and $(\G,D)$ has level at least $3$ then $w \geqslant 3$. Conversely, if $w$ is odd and at least $3$, by Convention \ref{conv} one has $h^{w,0} > 0$ and Proposition \ref{infhodgenum} gives that $h_{\inf}^{w} = \frac{1}{2}h^{w,0}(h^{w,0} + 1) > 0$. Therefore $(\G,D)$ has level at least $w$ hence at least $3$. If $w$ is even, in view of Proposition \ref{infhodgenum}, at least one of the conditions (2.a-c) is satisfied if, and only if $h_{\inf}^k > 0$ for some $k \geqslant 3$.
\end{proof}
Finally, we will need the following version of \cite[Lemme 3.3]{equllmo} whose proof we recall for completeness, to produce Hodge subdata of period data:
\begin{lemma}\label{crit}
Let $(\G,D)$ be a Hodge datum and $\M$ be a $\Q$-algebraic subgroup of $\G$. Assume that some $x \in D$ factors as \[x : \s \rightarrow \M_\R \rightarrow \G_\R.\] Let $D_M$ be the $\M(\R)$-orbit of $x$ in $\Hom_\R(\s, \M_\R)$. Then $(\M,D_M)$ is a Hodge subdatum of $(\G,D)$.
\end{lemma}
\begin{proof}
By assumption, $\M$ is a $\Q$-algebraic subgroup of $\G$ and post-composition with the inclusion maps $D_M$ inside of $D$. Therefore it suffices to prove that $(\M,D_M)$ is a Hodge datum, i.e. to check the axioms (HD$0$) and (HD$1$) of \cite[Def. 3.1]{geoao}. Let $w : \mult_{m,\R} \rightarrow \s$ be the weight homomorphism and $w_x := x \circ w$. As $(\G,D)$ is a Hodge datum, the morphism $w_x$ is defined over $\Q$ and has image in $Z(\G_\R)$ hence in $\M_\R \cap Z(\G_\R) \subset Z(\M_\R)$. This proves (HD$0$). Let $C = x(i)$. The Killing form is a $C$-polarization for the faithful representation of $\M^\ad(\R)$ on $\g_\R$. By \cite[(1.1.15)]{delshim}, the involution $\ad(C)$ is therefore a Cartan involution of $\M^\ad(\R)$. This proves (HD$1$).
\end{proof}
\subsection{The orthogonal case}
Let $(\G,D)$ be a period datum of orthogonal type associated to a tuple $(V,\psi, w = 2n, (h^{p,q}))$. The first part of Theorem \ref{computation} follows from Theorem \ref{main} and the following:
\begin{proposition}\label{hodgesuborth}
Assume $h^{n,n} \neq 0$. Then $(\G,D)$ has a Hodge subdatum $(\M,D_M)$ such that
\[
\dim D - \dim D_M = \sum_{i = 1}^n h^{n-i,n+i}.
\]
\end{proposition}
\begin{proof}
For $x \in D$, denote by
\[
V \otimes_\Q \C = \bigoplus_{p+q = 2n} V_x^{p,q}
\]
the associated Hodge decomposition. We first claim that there are some $x \in D$ and non-zero $v \in V$ such that $v \in V_x^{n,n}$. Indeed, let $x_0 \in D$ and pick $v_0 \in V_{x_0}^{n,n}$ a non-zero real vector, which is possible because $h^{n,n} \neq 0$ and $V_{x_0}^{n,n}$ is self-conjugate. Since $\G(\R)$ acts transitively on $V_\R - \{0\}$, there exists a $g \in \G(\R)$ such that $v := g(v_0) \in V-\{0\}$. Setting $x = g \cdot x_0$, the claim then simply follows from the fact that by definition of the action of $\G(\R)$ on $D$, one has $g(v_0) \in V_{g\cdot x_0}^{n,n}$. Let $\M = \stab_\G(v)$ which is a $\Q$-algebraic subgroup of $\G$, and $D_M = \M(\R) \cdot x$. As $v$ is a Hodge vector for $x$, one has for any $z \in \s(\R)$ that $x(z)(v) = |z|^n v$ so that, seen as a morphism of real algebraic groups $\s \rightarrow \G_\R$, $x$ factors through $\M_\R$. By Lemma \ref{crit}, this ensures that $(\M,D_M)$ is a Hodge subdatum of $(\G,D)$.

It remains to check that $D_M$ has the right codimension in $D$. For this, recall that as real-analytic manifolds $D \cong \G^\der(\R)/K$, that $\G^\der(\R) \cong \mathrm{SO}(r,s)$ as real Lie groups, and under this identification $K \cong \mathrm{SO}(h^{n,n}) \times \prod_{i = 1}^n \mathrm{U}(h^{n-i, n+i})$, where $r = \sum_{k} h^{n-2k-1, n+2k+1}$ and $s = \sum_k h^{n-2k, n+2k}$. Similarly $D_M \cong \M^\der(\R) / (K \cap \M^\der(\R))$ and under the above identification, $\M^\der(\R) \cong \mathrm{SO}(r,s-1)$ and $K\cap \M^\der(\R) \cong \mathrm{SO}(h^{n,n}-1) \times \prod_{i = 1}^n \mathrm{U}(h^{n-i,n+i})$. The dimension count is therefore as follows:
\begin{eqnarray*}
2 (\dim D - \dim D_M) & = & \dim_\R \mathrm{SO}(r,s) - \dim_\R \mathrm{SO}(r,s-1) \\
&& + \dim_\R \mathrm{SO}(h^{n,n}-1) \times \prod_{i = 1}^n \mathrm{U}(h^{n-i,n+i}) \\
&& -  \dim_\R \mathrm{SO}(h^{n,n}) \times \prod_{i = 1}^n \mathrm{U}(h^{n-i, n+i})\\
& = & \frac{d(d-1)}{2} - \frac{(d-1)(d-2)}{2} + \frac{(h^{n,n}-1)(h^{n,n}-2)}{2} \\ && - \frac{h^{n,n}(h^{n,n}-1)}{2} \\
& = & d-h^{n,n}. 
\end{eqnarray*}
where $d = r+s$ is the sum of all Hodge numbers.
\end{proof}
\begin{proof}[Proof of Theorem \ref{computation}(a)]
Let $\V$, $S$, $\ph$ and $(\G,D)$ be as in the statement. If the image of the period map has dimension $0$ there is nothing to prove. So we assume that $\dim \ph(S^\an) > 0$. Since $h^{n,n} \neq 0$ by assumption, Proposition \ref{hodgesuborth} shows that \[m_\hl(\G,D) \leqslant \Big(\sum_{i = 1}^n h^{n-i,n+i}\Big) - 1. \] Using Theorem \ref{main}, it remains to show that $\G^\der$ is $\Q$-simple and that $(\G,D)$ has level at least $3$. Because $\dim \ph(S^\an) > 0$, Griffiths transversality ensures that the infinitesimal Hodge number $h_{\inf}^1$ is non-zero. By Proposition \ref{infhodgenum}, there is a $0 \leqslant j \leqslant n-1$ such that $h^{2n-j,j}h^{2n-j-1, j+1} > 0$, and in particular $h^{2n-k,k} > 0$ for some $1 \leqslant k \leqslant n-1$. Summing up, we find that $h^{2n,0} = h^{0,2n} > 0$ by Convention \ref{conv}, that $h^{n,n} > 0$ by assumption, and that $h^{2n-k,k} = h^{k, 2n-k} > 0$. Since $n \geqslant 2$, this gives five distinct non-zero Hodge numbers so that $\dim_\Q V \geqslant 5$. This shows that $\G^\der = \mathbf{SO}(V,\psi)$ is $\Q$-simple by \cite[Prop. 2.14]{platrap} (the only non-simple case occurs for $\dim_\Q V = 4$). Finally, denote by $l = 2n-k$. We have:
\[
h^{2n,0}h^{2n-l,l} = h^{2n,0}h^{k, 2n-k} > 0.
\]
Since $n \geqslant 2$, we have $l \geqslant 3$ and we are either in case (2.b) or in case (2.c) of Corollary \ref{levelcrit}. This shows that $(\G,D)$ has level at least $3$, and this finishes the proof as explained above.
\end{proof}
\subsection{The symplectic case} Let $(\G,D)$ be a period datum of symplectic type associated to a tuple $(V,\psi, w = 2n+1, (h^{p,q}))$. The second part of Theorem \ref{computation} follows from Theorem \ref{main} and the following:
\begin{proposition}\label{hodgedatsymp}
Assume that $2d := \dim_\Q V > 2$ and let $h^{r,s}$ be a non-zero Hodge number. Then $(\G,D)$ has a Hodge subdatum $(\M,D_M)$ such that
\[
\dim D - \dim D_M =  \Big(2 \sum_{i = 0}^n h^{n-i, n+i+1}\Big) - h^{r,s} - 1.
\]
\end{proposition}
\begin{proof}
For $x \in D$, denote by
\[
V \otimes_\Q \C = \bigoplus_{p+q = 2n+1} V_x^{p,q}
\]
the associated Hodge decomposition. Since $\dim_\Q V > 0$, and the dimension of $V$ is the sum of the Hodge numbers, one of them is non-zero, say $h^{r,s}$ with $r>s$. Fix a non-zero vector $v_0 \in V_x^{r,s}$ and let $V_{0,\R}$ be the real vector subspace of $V_\R$ whose complexification is $V_{0,\C} := \C v_0 \oplus \C \overline{v_0}$. We first claim that we can choose $x \in D$ and $v_0 \in V_x^{r,s}$ so that $V_{0,\R}$ is the scalar extension to $\R$ of a $\Q$-vector subspace of $V$. Indeed, setting $2v_0^+ = v_0 + \overline{v_0}$ and $2iv_0^- = v_0 - \overline{v_0}$, one has by definition that $V_{0,\R}$ is the $\R$-span of $(v_0^-,v_0^+)$. Furthermore, because $\psi$ is a polarization of $x$, we have $\psi_\C(v_0,\overline{v_0}) \neq 0$ so that $\psi_\R(v_0^-, v_0^+) \neq 0$. In particular $(v_0^-,v_0^+)$ extends to a symplectic base of $(V_\R, \psi_\R)$. As $\G(\R)$ acts transitively on symplectic bases, there exists an element $g \in \G(\R)$ which sends $v_0^-$ and $v_0^+$ to rational vectors $v_1^-$ and $v_1^+$ satisfying $\psi(v_1^-, v_1^+) \neq 0$. Let $V_1$ be the span of $(v_1^-, v_1^+)$ and $v_1 =  v_1^+ + i v_1^-$ which belongs to $V_{x_1}^{n,n}$ where $x_1 = g \cdot x$. The pair $(x_1, v_1)$ satisfies the assumptions of the claim.

The assumption made on $\dim_\Q V$ implies that $V_1$ is a strict real vector subspace of $V$. Furthermore, we saw above that $\psi$ restricts to a non-degenerate skew-symmetric form on $V_1$. In particular, we get an orthogonal decomposition
\[
V = V_1 \oplus V_1^\bot
\]
into non-zero rational polarized Hodge substructures. Let $\M$ be the $\Q$-algebraic subgroup of $\G$ preserving this decomposition, which is the fixator in $\G$ of the projection on $V_1$ along $V_1^\bot$ (through the natural action of $\G$ on $\End(V)$ by conjugation). We claim that $x$ factors as $\s \rightarrow \M_\R \hookrightarrow \G_\R$. To see that, we need to prove that for $z \in \s(\R)$, the subspace $V_{1,\R}$ is preserved by $x(z)$ (this is enough to ensure that $x(z) \in \M_\R$ because $x(z)$ respects $\psi_\R$ up to scaling). Let $v \in V_1$ and decompose it as $v = \lambda v_1 + \overline{\lambda} \overline{v_1}$. Then
\[
x(z)(v) = \lambda z^r \overline{z}^s v_1 + \overline{\lambda} z^s\overline{z}^r \overline{v_1} \in V_1,
\]
which proves the claim. Then, $(\M,D_M)$ is a Hodge subdatum of $(\G,D)$ by Lemma \ref{crit}

%
We are now left with computing dimensions. We have $D \cong \G^\der(\R)/K$ as real-analytic manifolds where $\G^\der(\R) \cong \mathrm{Sp}_{2d}(\R)$ as real Lie groups and under this identification $K \cong \prod_{i=0}^n \mathrm{U}(h^{n+i+1,n-i})$. Similarly, $D_M \cong \M^\der(\R)/(K \cap \M^\der(\R))$ and under the above identifications of real Lie groups, we have $\M^\der(\R) \cong \mathrm{Sp}_{2d-2}(\R)\times\mathrm{Sp}_2(\R)$ and $K \cap \M^\der(\R) \cong \mathrm{U}(1) \times \mathrm{U}(h^{r,s}-1)\times \prod_{0 \leqslant i \leqslant n, i \neq n-s}\mathrm{U}(h^{n+i+1, n-i})$. The dimension count is therefore as follows:
\begin{eqnarray*}
2 (\dim D - \dim D_M) & = & \dim_\R \mathrm{Sp}_{2d}(\R) - \dim_\R \mathrm{Sp}_{2d-2}(\R)\times\mathrm{Sp}_2(\R) \\
&& + \dim_\R \mathrm{U}(1) \times \mathrm{U}(h^{r,s}-1)\times \prod_{0 \leqslant i \leqslant n, i \neq n-s}\mathrm{U}(h^{n+i+1, n-i}) \\
&& - \dim_\R \prod_{i=0}^n \mathrm{U}(h^{n+i+1,n-i}) \\
& = & d(2d+1) - (3 + (d-1)(2d-1)) + 1 + (h^{r,s}-1)^2 - (h^{r,s})^2 \\
& = & 4d - 2 - 2h^{r,s}. 
\end{eqnarray*}
The identity $d = \sum_{i=0}^{n} h^{n+i+1,n-i}$ then gives the result.
\end{proof}
\begin{proof}[Proof of Theorem \ref{computation}(b)]
Let $\V$, $S$, $\ph$ and $(\G,D)$ be as in the statement. If the image of the period map has dimension $0$ there is nothing to prove. So we assume that $\dim \ph(S^\an) > 0$. By \cite[Prop. 2.13]{platrap}, the derived subgroup $\G^\der = \mathbf{Sp}(V,\psi)$ is $\Q$-simple. By assumption, the weight $2n+1$ is at least $3$ so by Corollary \ref{levelcrit} the Hodge datum $(\G,D)$ has level at least $3$. 

Because $\dim \ph(S^\an) > 0$, Griffiths transversality ensures that the infinitesimal Hodge number $h_{\inf}^1$ is non-zero. This implies that $h^{2n+1-k, k} > 0$ for some $1 \leqslant k \leqslant n$. As $n \geqslant 1$, this gives four distinct non-zero Hodge numbers $h^{2n+1,0} = h^{0,2n+1}$ and $h^{2n+1-k,k} = h^{k, 2n+1-k}$. In particular $\dim_\Q V \geqslant 4$.  By Proposition \ref{hodgedatsymp} we therefore have that for any non-zero Hodge number $h^{r,s}$,
\[
m_\hl(\G,D) \leqslant \Big(2 \sum_{i = 0}^n h^{n-i, n+i+1}\Big) - 2 - h^{r,s}.
\]
Applying Theorem \ref{main} finishes the proof.
\end{proof}
\section{Maximality of the universal variation of non-singular sextic fourfolds}\label{sharp}
Let $n \geqslant 4$ be an even integer, $U_n = \p(H^0(\p^{n+1}(\C), \mathcal{O}(n+2))) - \Delta_n$ the parameter space of smooth hypersurfaces of degree $d = n+2$ in $\p^{n+1}(\C)$. Let $f_n : \mathcal{X}_n \rightarrow U_n$ be the universal family of degree $d$ hypersurfaces in $\p^{n+1}(\C)$, which is explicitly given by: 
\[
f_n : \mathcal{X}_n = \Big\{(x,u) \in \p^{n+1}(\C) \times U_n \hspace{0.1cm} : \hspace{0.1cm} x \in X_u \Big\} \subset \p^{n+1}(\C) \times U_n \twoheadrightarrow U_n. 
\]
Here for $u \in U_n$, we denoted by $X_u$ the associated hypersurface in $\p^{n+1}(\C)$. Let $\V_n = (R^n f_\e \underline{\Z})_\prim/ \mathrm{torsion}$ be the variation on primitive middle cohomology associated to $f_n$. Let $(\G_n, D_n)$ be the generic Hodge datum of $\V_n$, fix $\Gamma_n\subset \G_n(\Q)$ an arithmetic lattice containing the image of the monodromy representation associated to $\V_n$ and $\ph_n : U_n \rightarrow \Gamma_n \quo D_n$ the associated period map. By Picard-Lefschetz's formula, $(\G_n, D_n)$ is a period datum (\cite[Cor. 5.5]{weil}) of othogonal type (because we assumed that $n$ is even). Furthermore, it has non-zero middle Hodge number so that, by Proposition \ref{hodgesuborth}, one has
\[
m_\hl(\G_n, D_n) \leqslant h^{\frac{n}{2}+1, \frac{n}{2}-1} + \cdots + h^{n,0} - 1.
\]
Here we denoted by $(h^{p,q})_{p+q = n}$ the Hodge numbers of the polarized Hodge structure on the primitive part of the middle cohomology of any smooth hypersurface parametrized by $U_n$. Proposition \ref{sharpfourfolds} then follows from the $n = 4$ case of the following computation of the dimension of the image of $\ph_n$:
\begin{proposition}
The image of $\ph_n$ has dimension $h^{n-1,1}$.
\end{proposition}
\begin{proof}
Fix a smooth point $x \in \ph_n(U_n)$ and some $u \in U_n$ such that $x = \ph_n(u)$. We need to prove that $\dim_\C \im(d_u\ph_n) = h^{n-1,1}$. Recall that by \cite[Thm. 1.23]{grifper2} the differential at $u$ of the period map can be seen as a $\C$-linear map
\[
d_u \ph_n : T_u U_n \rightarrow \bigoplus_p \Hom(H^{n-p}(X_u, \Omega_{X_u}^p)_\prim, H^{n-p+1}(X_u, \Omega_{X_u}^{p-1})_\prim)
\]
and is given explicitly as the composition of the Kodaira-Spencer map at $u$ associated to $f_n$
\[
\kappa_u : T_u U_n \rightarrow H^1(X_u, T_{X_u})
\]
with the direct sum $\mu$ of the maps \[\mu_p : H^1(X_u, T_{X_u}) \rightarrow  \Hom(H^{n-p}(X_u, \Omega_{X_u}^p)_\prim, H^{n-p+1}(X_u, \Omega_{X_u}^{p-1})_\prim)\] induced in cohomology by the interior products $T_{X_u} \otimes \Omega_{X_u}^p \rightarrow \Omega_{X_u}^{p-1}$. By \cite[Lemme 18.15]{voi} the map $\kappa_u$ is surjective since we assumed $n \geqslant 4$ (note that the $n$ in \textit{op. cit.} is our $n+1$). Furthermore, recall the:
\begin{lemma}
The hypersurface $X_u$ has trivial canonical bundle $K_{X_u}$.
\end{lemma}
\begin{proof}
This is well-known consequence of the conormal exact sequence for the inclusion of $X_u$ in $\p^{n+1}(\C)$.
\end{proof}
In particular, the choice of a generator $\Omega \in H^0(X_u, K_{X_u})$ produces an isomorphism $T_{X_u} \cong \Omega_{X_u}^{n-1}$ given by the interior product with $\Omega$. Denote by
\[
[\Omega] : H^1(X_u,T_{X_u}) \cong H^1(X_u, \Omega_{X_u}^{n-1})
\]
the isomorphism induced in cohomology. In particular we have $\dim_\C H^1(X_u, T_{X_u}) = h^{n-1,1}$ and we are left with proving that the map $\mu$ is injective. For this, it suffices to show that its $n$-th component 
\[
\mu_n : H^1(X_u, T_{X_u}) \rightarrow  \Hom(H^{0}(X_u, \Omega_{X_u}^n), H^{1}(X_u, \Omega_{X_u}^{n-1}))
\]
is. But by definition, this map sends a class $\alpha$ to the unique $\C$-linear map sending the generator $\Omega$ to the class $[\Omega](\alpha)$. In particular it is an isomorphism and $\mu$ is therefore injective. This finishes the proof.
\end{proof}
\printbibliography
\bigskip
\noindent
\small{\textsc{I.H.E.S., Université Paris-Saclay, Laboratoire Alexander Grothendieck, 35 Route De Chartres, 91440 Bures-Sur-Yvette (France)}\\
\textit{E-mail :} \texttt{khelifa@ihes.fr}}
\end{document}